\documentclass[11pt, a4paper]{article}
\usepackage{mathrsfs, mathtools, amsthm, amssymb, thm-restate}  
\usepackage{paralist, tablists}  

\usepackage{graphicx, xcolor, tikz}  
\usetikzlibrary{calc, shapes, decorations.pathreplacing, decorations.markings}  
\usepackage{caption}  
\usepackage[labelformat=simple]{subcaption}  

\usepackage[left=25mm, top=25mm, bottom=25mm, right=25mm]{geometry}  
\usepackage[T1]{fontenc} 

\usepackage[inline]{enumitem}  
\usepackage[square, numbers, sort&compress]{natbib}  
\usepackage[
  bookmarks=true,                
  bookmarksnumbered=true,        
  bookmarksopen=true,            
  plainpages=false,              
  colorlinks=true,               
  linkcolor=blue,                
  citecolor=black,               
  anchorcolor=green,             
  urlcolor=blue,                 
  hyperindex=true                
]{hyperref} 

\newtheorem{theorem}{Theorem}[section] 
\newtheorem{lemma}[theorem]{Lemma}

\newtheorem{conjecture}{Conjecture}
\newtheorem{case}{Case}

\newtheorem{claim}{Claim}
\theoremstyle{definition}

\newtheorem{proposition}{Proposition}

\makeatletter
\def\th@plain{%
  \upshape 
}
\makeatother

\newcommand{\ie}{i.e.,\ }

\usepackage{marvosym,wasysym}
\usepackage{array}
\usepackage{bm}  

\makeatletter
\renewenvironment{proof}[1][\proofname]{\par
  \pushQED{\qed}%
  \normalfont \topsep6\p@\@plus6\p@\relax
  \trivlist
  \item[\hskip\labelsep
        \bfseries
    #1\@addpunct{.}]\ignorespaces
}{%
  \popQED\endtrivlist\@endpefalse
}
\makeatother

\usepackage[capitalise]{cleveref}  
\crefname{claim}{Claim}{Claims}

\newcommand{\solidnode}[1]{\node[circle, inner sep = 1, fill = black, draw] () at (#1) {}}

\begin{document}
\title{Claw-free cubic graphs are (1, 1, 1, 3)-packing edge-colorable}
\author{Jingxi Hou\footnote{School of Mathematics and Statistics, Henan University, Kaifeng, 475004, P. R. China.} \and Tao Wang\footnote{Center for Applied Mathematics, Henan University, Kaifeng, 475004, P. R. China. \texttt{Corresponding author: wangtao@henu.edu.cn; https://orcid.org/0000-0001-9732-1617}. Tao Wang was supported by the Natural Science Foundation of Henan Province (No. 242300420238).} \and Xiaojing Yang\footnote{School of Mathematics and Statistics, Henan University, Kaifeng, 475004, P. R. China. Xiaojing Yang was supported by National Natural Science Foundation of China (No. 12101187).}}
\date{July 19, 2025}
\maketitle

\begin{abstract}
For a non-decreasing positive integer sequence $S = (s_{1}, \dots, s_{k})$, an $S$-packing edge-coloring of a graph $G$ is a partition of the edge set of $G$ into subsets $E_{1}, \dots, E_{k}$ such that for each $1 \leq i \leq k$, the distance between any two distinct edges $e_{1}, e_{2} \in E_{i}$ is at least $s_{i} + 1$. Gastineau and Togni conjectured that cubic graphs, except the Petersen and Tietze graphs, admit $(1, 1, 1, 3)$-packing edge-colorings. In this paper, we prove that every claw-free cubic graph admits such a coloring.

\end{abstract}

\section{Introduction}
A \emph{proper edge-coloring} of a graph $G$ is an assignment of colors to its edges such that no two adjacent edges receive the same color. This concept can be extended to \emph{$d$-distance edge-coloring}, where the \emph{edge-distance} in $G$, defined as the vertex-distance in its line graph, must be at least $d + 1$ for edges sharing the same color. 

A \emph{$d$-packing edge-coloring} of $G$ partitions $E(G)$ into color classes such that the distance between any two edges within the same class is at least $d + 1$. Generalizing this notion, given a non-decreasing sequence $S = (s_{1}, \dots, s_{k})$, an \emph{$S$-packing edge-coloring} assigns colors to edges such that any two edges sharing color $i$ have a distance greater than $s_{i}$. A graph admitting such a coloring is called \emph{$S$-packing edge-colorable}. Introduced by Gastineau and Togni \cite{MR3944589} as an extension of $S$-packing coloring (proposed by Goddard and Xu \cite{MR3163180, MR2993523}), this concept has inspired significant research (see \cite{MR4878746, MR3600792, MR4145729, MR4273008, MR3508758}). 

In an $S$-packing edge-coloring, each subset $E_{i}$ is referred to as an \emph{$s_{i}$-packing}. Notably, a $1$-packing corresponds to a matching, while a $2$-packing forms an induced matching. For brevity, exponents denote repeated elements in $S$; e.g., $(1, 2^{4})$ represents $(1, 2, 2, 2, 2)$. Hocquard, Lajou and Lu\v{z}ar \cite{MR4520054} proved that all subcubic graphs admit $(1^{2}, 2^{5})$- and $(1, 2^{8})$-packing edge-colorings. Liu, Santana and Short \cite{MR4660817} confirmed that all subcubic multigraphs are $(1, 2^{7})$-packing edge-colorable. Gastineau and Togni \cite{MR3944589} showed that cubic graphs with a $2$-factor are $(1^{3},3^{2})$-, $(1^{3}, 4^{5})$-, and $(1^{2}, 2^{5})$-packing edge-colorable, where a $2$-factor is a $2$-regular spanning subgraph. Li, Li and Liu \cite{Li2024} established that subcubic outerplanar graphs admit $(1, 2^{5})$- and $(1^{2}, 2^{3})$-packing edge-colorings, but exclude $(1, 2^{4})$- or $(1^{2}, 2^{2})$-packing edge-colorings. Liu and Yu \cite{Liu2024} further showed that all subcubic graphs are $(1^{2}, 2^{4})$-packing edge-colorable.

Gastineau and Togni \cite{MR3944589} also studied $S$-packing edge-colorings of subcubic graphs with specific 1's in $S$. By Vizing's theorem, every subcubic graph is $(1, 1, 1, 1)$-packing edge-colorable. Additionally, the following theorem follows from the results of Payan and of Fouquet and Vanherpe:

\begin{theorem}[Payan \cite{Payan1977}, Fouquet and Vanherpe \cite{MR3053595}]
Every subcubic graph admits a $(1, 1, 1, 2)$-packing edge-coloring.
\end{theorem}

The fact that the Petersen graph and the Tietze graph are not $(1, 1, 1, 3)$-packing edge-colorable establishes that $2$ cannot be replaced by $3$ in the above theorem. Gastineau and Togni \cite{MR3944589} proposed the following conjecture.

\begin{conjecture}[Gastineau and Togni \cite{MR3944589}]
Every cubic graph, except the Petersen and Tietze graphs, is $(1, 1, 1, 3)$-packing edge-colorable.
\end{conjecture}

This conjecture has been verified for certain restricted classes of snarks. In this paper, we extend this line of research by focusing on claw-free cubic graphs and proving the following result. A graph is \emph{claw-free} if it contains no $K_{1, 3}$ as an induced subgraph.

\begin{theorem}\label{MainResult}
Every claw-free cubic graph is $(1, 1, 1, 3)$-packing edge-colorable.
\end{theorem}

\section{Proof of \cref{MainResult}}
A \emph{diamond} is a graph $D$ formed by removing one edge from the complete graph $K_{4}$. In a diamond $D$, the two degree-3 vertices are called \emph{internal vertices}, while the remaining two are \emph{external vertices}.

For a positive integer $k$, a \emph{string of $k$ diamonds} is constructed as follows: take $k$ disjoint diamonds $D_{1}, \dots, D_{k}$, where $V(D_{i}) = \{x_{i}, y_{i}, z_{i}, w_{i}\}$, and $z_{i}, w_{i}$ are the internal vertices of $D_{i}$, add two additional isolated vertices $x_{0}$ and $y_{k+1}$, and connect $x_{i-1}$ to $y_{i}$ for all $i \in [k + 1]$. A \emph{ring of diamonds} is a connected claw-free cubic graph in which every vertex belongs to a diamond. Note that a ring of diamonds can be obtained from a string of diamonds by removing $x_{0}$ and $y_{k+1}$, and adding an edge between $y_{1}$ and $x_{k}$.

For a graph $G$, \emph{replacing an edge with a string of $k$ diamonds} is defined as follows: for an edge $xy \in E(G)$, remove $xy$ from $G$, and identify $x_{0}$ of a string of $k$ diamonds with $x$ and $y_{k+1}$ of the same string with $y$.

We rely on the following two results in our proof.

\begin{theorem}[Oum \cite{MR2788679}]\label{thm:oum}
A graph $G$ is $2$-edge-connected, claw-free, and cubic if and only if it satisfies one of the following:
\begin{enumerate}[label = (\roman*)]
    \item $G \cong K_{4}$;
    \item $G$ is a ring of diamonds; or
    \item $G$ is constructed from a $2$-edge-connected cubic multigraph $H$ by
    \begin{itemize}
        \item substituting each vertex of $H$ with a triangle (see \cref{substitute}), and/or
        \item replacing some edges of $H$ with strings of diamonds.
    \end{itemize}
\end{enumerate}
\end{theorem}

\begin{figure}
    \centering
    \begin{tikzpicture}[line width=1pt, scale=1, every node/.style={inner sep=1pt}]
        \def\r{6}
        \def\coreangle{120}
        \def\subangle{60}

        \foreach \i/\name in {90/A1, 210/A2, 330/A3} {
            \coordinate (\name) at (\i:1);
            \draw (0,0) -- (\name);
        }

        \foreach \name/\bname/\cname in {
            A1/B1/C1,
            A2/B2/C2,
            A3/B3/C3} {
            \coordinate (\bname) at ($(\name) + (\r,0)$);
            \coordinate (\cname) at ($(\r,0)!0.4!(\bname)$);
            \draw (\bname) -- (\cname);
        }

        \draw (C1) -- (C2) -- (C3) -- cycle;

        \draw[->] (2,0) -- ({\r - 2},0);
        \node at ({\r/2}, 0.2) {Substitute};

        \foreach \pt in {C1,C2,C3} {
            \solidnode{\pt};
        }
        \solidnode{0,0};
    \end{tikzpicture}
    \caption{Substituting a vertex.}
    \label{substitute}
\end{figure}

\begin{theorem}[Plesn\'{i}k \cite{MR317999}]\label{thm:2factor}
Let $G$ be a $2$-edge-connected cubic multigraph. For any two edges $e, f \in E(G)$, the graph $G$ contains a $2$-factor that includes $e$ and $f$.
\end{theorem}

In this paper, we assign colors to the edges of $G$ from the set $C = \{ 1_{a}, 1_{b}, 1_{c}, 3_{a} \}$.

\begin{proposition}\label{prop:ring}
If $G$ is a ring of diamonds, then $G$ is $(1, 1, 1)$-packing edge-colorable.
\end{proposition}

\begin{proof}
In each diamond, there are two sets of non-adjacent external edges and one internal edge. For a ring of diamonds, we assign colors as follows:
\begin{itemize}
    \item color one set of non-adjacent external edges with $1_{a}$,
    \item color the remaining set of non-adjacent external edges with $1_{b}$,
    \item color the internal edges of each diamond, as well as the edges connecting adjacent diamonds, with $1_{c}$.
\end{itemize}
This ensures that edges sharing the same color are pairwise non-adjacent (for $1_{a}$ and $1_{b}$) or separated by at least one edge (for $1_{c}$), satisfying the distance requirement for a $(1, 1, 1)$-packing edge-coloring. See \cref{Ring} for an illustration of this coloring scheme applied to a ring of diamonds.
\end{proof}

\begin{figure}
    \centering
    \begin{tikzpicture}[line width=1pt, node distance=2cm, scale=1]
        \tikzset{
            rededge/.style={red},
            greenedge/.style={green},
            blueedge/.style={blue},
        }

        \newcommand{\drawDiamond}[2]{
            \coordinate (A#2) at (#1, 0);
            \coordinate (B#2) at (#1 + 2, 0);
            \coordinate (C#2) at (#1 + 1, 1);
            \coordinate (D#2) at (#1 + 1, -1);
            
            \draw[rededge] (A#2) -- (D#2);
            \draw[greenedge] (A#2) -- (C#2);
            \draw[rededge] (B#2) -- (C#2);
            \draw[greenedge] (B#2) -- (D#2);
            \draw[blueedge] (C#2) -- (D#2);
        }

        \drawDiamond{0}{1}
        \drawDiamond{4}{2}
        \drawDiamond{8}{3}

        \draw[blueedge] (B1) -- (A2);  
        \draw[blueedge] (B2) -- (A3);  
        \draw[blueedge, bend left=120] (A1) to (B3); 
        \newcommand{\drawnode}[1]{
        \solidnode{A#1};
        \solidnode{B#1};
        \solidnode{C#1};
        \solidnode{D#1};
        }
        \drawnode{1};
        \drawnode{2};
        \drawnode{3};
    \end{tikzpicture}
    \caption{A $(1, 1, 1)$-packing edge-coloring of a ring of diamonds}
    \label{Ring}
\end{figure}

\begin{theorem}\label{m}
If $G$ is a $2$-edge-connected claw-free cubic graph, then $G$ is $(1, 1, 1, 3)$-packing edge-colorable.
\end{theorem}

\begin{proof}
By \cref{thm:oum}, $G$ satisfies one of the following: 
\begin{enumerate}[label = (\roman*)]
    \item $G \cong K_{4}$;
    \item $G$ is a ring of diamonds; or
    \item $G$ is constructed from a $2$-edge-connected cubic multigraph $H$ by
    \begin{itemize}
        \item substituting each vertex of $H$ with a triangle, and/or
        \item replacing some edges of $H$ with strings of diamonds.
    \end{itemize}
\end{enumerate}

If $G \cong K_{4}$, then $G$ is trivially $(1, 1, 1)$-packing edge-colorable. If $G$ is a ring of diamonds, by \cref{prop:ring}, $G$ admits a $(1, 1, 1)$-packing edge-coloring. Now, consider the case where $G$ is constructed by substituting each vertex of a $2$-edge-connected, cubic multigraph $H$ with a triangle, and/or replacing some edges of $H$ with strings of diamonds. (Note that edges may or may not be replaced.) By \cref{thm:2factor}, $H$ contains a $2$-factor $\mathcal{F}'$ that is a disjoint union of cycles, say $\mathcal{F}' = C_{1}' \cup \dots \cup C_{k}'$.

\noindent\textbf{Case 1: No edges are replaced by strings of diamonds.}

Suppose $G$ is constructed by substituting each vertex of $H$ with a triangle but without replacing any edge of $H$ with strings of diamonds. For each $i \in [k]$, let $C_{i}': h_{1}^{i}h_{2}^{i} \dots h_{m_{i}}^{i}h_{1}^{i}$ be a cycle in $\mathcal{F}'$. Denote by $C_{i}$ the subgraph of $G$ with a Hamiltonian cycle $x_{1}^{i}h_{1}^{i}y_{1}^{i}x_{2}^{i}h_{2}^{i}y_{2}^{i} \dots x_{m_{i}}^{i}h_{m_{i}}^{i}y_{m_{i}}^{i}x_{1}^{i}$, where $x_{j}^{i}h_{j}^{i}y_{j}^{i}$ represents the substitution of the vertex $h_{j}^{i}$ in $H$ with the triangle $x_{j}^{i}h_{j}^{i}y_{j}^{i}$.

Observe that the edge $y_{j}^{i}x_{j+1}^{i}$ on $C_{i}$ of $G$ corresponds to the edge $h_{j}^{i} h_{j+1}^{i}$ on $C_{i}'$ of $H$. Let $\mathcal{F} = C_{1} \cup \dots \cup C_{k}$ (see \cref{M1}, which illustrates the construction of $G$ by substituting each vertex of $H$ with a triangle). Since $\mathcal{F}'$ is a $2$-factor of $H$, it follows that $\mathcal{F}$ becomes a $2$-factor of $G$, covering all vertices of $G$ with disjoint cycles.

\begin{figure}
    \centering
    \begin{tikzpicture}[line width=1pt, scale=0.12, every node/.style={font=\small}]

    \tikzset{
        triangle/.style={draw},
        diamond/.style={draw}
    }

    \newcommand{\drawTriangle}[2]{
        \coordinate (h3#2) at (#1, 0);
        \coordinate (h2#2) at (#1 + 20, 0);
        \coordinate (h1#2) at ($(h3#2)!1!60:(h2#2)$);
        \coordinate (x1#2) at ($(h1#2)!0.25!(h3#2)$);
        \coordinate (y1#2) at ($(h1#2)!0.25!(h2#2)$);
        \coordinate (x2#2) at ($(h1#2)!0.75!(h2#2)$);
        \coordinate (y2#2) at ($(h2#2)!0.25!(h3#2)$);
        \coordinate (x3#2) at ($(h2#2)!0.75!(h3#2)$);
        \coordinate (y3#2) at ($(h1#2)!0.75!(h3#2)$);
        
        \draw[triangle] (h3#2) -- (h2#2) -- (h1#2) -- cycle;

        \node at (h1#2) [above] {$h_{1}^{#2}$};
    }

    \newcommand{\drawDiamond}[1]{
        \coordinate (h1#1) at ($(h21)!0.5!(h32) + (0, -10)$);
        \coordinate (h3#1) at ($(h1#1) + (0,-20)$);
        \coordinate (h2#1) at ($(h1#1)!1!60:(h3#1)$);
        \coordinate (h4#1) at ($(h1#1)!1!-60:(h3#1)$);
        
        \coordinate (x1#1) at ($(h1#1)!0.25!(h4#1)$);
        \coordinate (y1#1) at ($(h1#1)!0.25!(h2#1)$);
        \coordinate (x2#1) at ($(h1#1)!0.75!(h2#1)$);
        \coordinate (y2#1) at ($(h2#1)!0.25!(h3#1)$);
        \coordinate (x3#1) at ($(h2#1)!0.75!(h3#1)$);
        \coordinate (y3#1) at ($(h3#1)!0.25!(h4#1)$);
        \coordinate (x4#1) at ($(h3#1)!0.75!(h4#1)$);
        \coordinate (y4#1) at ($(h1#1)!0.75!(h4#1)$);

        \draw (h1#1) -- (h3#1);
        \draw (h1#1) -- (h2#1) -- (h3#1) -- (h4#1) -- cycle;
        
        \node at (h1#1) [above] {$h_{1}^{#1}$};
        \node at (h2#1) [below] {$h_{2}^{#1}$};
        \node at (h3#1) [below] {$h_{3}^{#1}$};
        \node at (h4#1) [below] {$h_{4}^{#1}$};
    }

\begin{scope}[shift={(-35, 0)}]
    \drawTriangle{0}{1}
    \drawTriangle{32}{2}

    \node at (h21) [below] {$h_{2}^{1}$};
    \node at (h31) [left] {$h_{3}^{1}$};
    \node at (h22) [right] {$h_{2}^{2}$};
    \node at (h32) [below] {$h_{3}^{2}$};
    
    \draw (h11) -- (h12);
    \draw (h21) -- (h32);

    \drawDiamond{3}

    \draw [bend right=10] (h31) to (h43);
    \draw [bend left=10] (h22) to (h23);
    \node at ($(h33) + (0, -10)$) {\textbf{$H$}};
    
        \foreach \i in {1, 2, 3}{
           \foreach \j in {1, 2, 3}{
              \solidnode{h\i\j};
            }
    }
              \solidnode{h43};
\end{scope}
\begin{scope}[shift={(35, 0)}]
    \drawTriangle{0}{1}
    \drawTriangle{32}{2}
    
    \newcommand{\drawXY}[1]{
        \draw (x1#1) -- (y1#1)
                     (x2#1) -- (y2#1)
                     (x3#1) -- (y3#1);

        \node at (x1#1) [left] {$x_{1}^{#1}$};
        \node at (y1#1) [right] {$y_{1}^{#1}$};
        \node at (x2#1) [right] {$x_{2}^{#1}$};
        \node at (y2#1) [below] {$y_{2}^{#1}$};
        \node at (x3#1) [below] {$x_{3}^{#1}$};
        \node at (y3#1) [left] {$y_{3}^{#1}$};
    }
    \drawXY{1}
    \drawXY{2}
    
    \node at (h21) [below] {$h_{2}^{1}$};
    \node at (h31) [left] {$h_{3}^{1}$};
    \node at (h22) [right] {$h_{2}^{2}$};
    \node at (h32) [below] {$h_{3}^{2}$};
    
    \draw (h11) -- (h12);
    \draw (h21) -- (h32);

    \drawDiamond{3}
    \draw (x13) -- (y13)
                 (x23) -- (y23)
                 (x33) -- (y33)
                 (x43) -- (y43);
    \node at (x13) [above] {$x_{1}^{3}$};
    \node at (y13) [above] {$y_{1}^{3}$};
    \node at (x23) [above] {$x_{2}^{3}$};
    \node at (y23) [below] {$y_{2}^{3}$};
    \node at (x33) [below] {$x_{3}^{3}$};
    \node at (y33) [below] {$y_{3}^{3}$};
    \node at (x43) [below] {$x_{4}^{3}$};
    \node at (y43) [above] {$y_{4}^{3}$};
    
    \draw [bend right=10] (h31) to (h43);
    \draw [bend left=10] (h22) to (h23);
    \node at ($(h33) + (0, -10)$) {\textbf{$G$}};
    
        \foreach \i in {1, 2, 3}{
           \foreach \j in {1, 2, 3}{
              \solidnode{h\i\j};
              \solidnode{x\i\j};
              \solidnode{y\i\j};
            }
    }
              \solidnode{h43};
              \solidnode{x43};
              \solidnode{y43};
\end{scope}
    \end{tikzpicture}
    \caption{Creating $G$ by substituting each vertex of $H$ with a triangle.}
    \label{M1}
\end{figure}

Since $H - E(\mathcal{F}')$ is a perfect matching in $H$, denote this matching by $M'$. (For example, in \cref{M1}, $M' = \{ h_{1}^{1}h_{1}^{2}, h_{2}^{1}h_{3}^{2}, h_{3}^{1}h_{4}^{3}, h_{1}^{3}h_{3}^{3}, h_{2}^{3}h_{2}^{2} \}$.) Define $M$ as the set of edges obtained by adding all edges $x_{j}^{i}y_{j}^{i}$ (from the substituted triangles) to $M'$. Consequently, $G$ decomposes into $M$ and $\mathcal{F}$, \ie $G = M + \mathcal{F}$.

We consider two cases based on the parity of $m_{i}$:
\begin{itemize}
    \item[(1)] When $m_{i}$ is even, color the edges of $C_{i}$ alternately with $1_{a}$ and $1_{b}$ in sequence.
    \item[(2)] When $m_{i}$ is odd, select an edge $y_{j}^{i}x_{j+1}^{i}$ in $C_{i}$ and color it with $3_{a}$. Then, color the remaining edges along the path $x_{j+1}^{i}h_{j+1}^{i} \dots x_{j}^{i}h_{j}^{i}y_{j}^{i}$ alternately with $1_{a}$ and $1_{b}$.
\end{itemize}

Finally, color all edges in $M$ with $1_{c}$. It is straightforward to verify that the resulting coloring is indeed a $(1, 1, 1, 3)$-packing edge-coloring of $G$. (See \cref{M2} for an illustration of a $(1, 1, 1, 3)$-packing edge-coloring of $G$ after substituting each vertex of $H$ with a triangle.)

\begin{figure}
    \centering
    \begin{tikzpicture}[line width=1pt, scale=0.12, every node/.style={font=\small}]

    \newcommand{\drawTriangle}[2]{
        \coordinate (h3#2) at (#1, 0);
        \coordinate (h2#2) at (#1 + 20, 0);
        \coordinate (h1#2) at ($(h3#2)!1!60:(h2#2)$);
        \coordinate (x1#2) at ($(h1#2)!0.25!(h3#2)$);
        \coordinate (y1#2) at ($(h1#2)!0.25!(h2#2)$);
        \coordinate (x2#2) at ($(h1#2)!0.75!(h2#2)$);
        \coordinate (y2#2) at ($(h2#2)!0.25!(h3#2)$);
        \coordinate (x3#2) at ($(h2#2)!0.75!(h3#2)$);
        \coordinate (y3#2) at ($(h1#2)!0.75!(h3#2)$);
        
    \draw [brown] (y1#2) -- (x2#2);
        \draw [red] (x2#2) -- (h2#2)
                              (y2#2) -- (x3#2)
                              (h3#2) -- (y3#2)
                              (x1#2) -- (h1#2);
    \draw [green] (h1#2) -- (y1#2)
                              (h2#2) -- (y2#2)
                              (x3#2) -- (h3#2)
                              (y3#2) -- (x1#2);
       
        \draw [blue] (x1#2) -- (y1#2)
                                (x2#2) -- (y2#2)
                                (x3#2) -- (y3#2);

        \node at (h1#2) [above] {$h_{1}^{#2}$};
        \node at (x1#2) [left] {$x_{1}^{#2}$};
        \node at (y1#2) [right] {$y_{1}^{#2}$};
        \node at (x2#2) [right] {$x_{2}^{#2}$};
        \node at (y2#2) [below] {$y_{2}^{#2}$};
        \node at (x3#2) [below] {$x_{3}^{#2}$};
        \node at (y3#2) [left] {$y_{3}^{#2}$};
    }

    \newcommand{\drawDiamond}[1]{
        \coordinate (h1#1) at ($(h21)!0.5!(h32) + (0, -10)$);
        \coordinate (h3#1) at ($(h1#1) + (0,-20)$);
        \coordinate (h2#1) at ($(h1#1)!1!60:(h3#1)$);
        \coordinate (h4#1) at ($(h1#1)!1!-60:(h3#1)$);
        
        \coordinate (x1#1) at ($(h1#1)!0.25!(h4#1)$);
        \coordinate (y1#1) at ($(h1#1)!0.25!(h2#1)$);
        \coordinate (x2#1) at ($(h1#1)!0.75!(h2#1)$);
        \coordinate (y2#1) at ($(h2#1)!0.25!(h3#1)$);
        \coordinate (x3#1) at ($(h2#1)!0.75!(h3#1)$);
        \coordinate (y3#1) at ($(h3#1)!0.25!(h4#1)$);
        \coordinate (x4#1) at ($(h3#1)!0.75!(h4#1)$);
        \coordinate (y4#1) at ($(h1#1)!0.75!(h4#1)$);

        \draw [blue] (h1#1) -- (h3#1);
        \draw (h1#1) -- (h2#1) -- (h3#1) -- (h4#1) -- cycle;
        \draw [red]   (x1#1) -- (h1#1)
                            (y1#1) -- (x2#1)
                            (h2#1) -- (y2#1)
                            (x3#1) -- (h3#1)
                            (y3#1) -- (x4#1)
                            (h4#1) -- (y4#1);
                            
      \draw [green] (h1#1) -- (y1#1)
                            (x2#1) -- (h2#1)
                            (y2#1) -- (x3#1)
                            (h3#1) -- (y3#1)
                            (x4#1) -- (h4#1)
                            (y4#1) -- (x1#1);
                            
        \draw [blue] (x1#1) -- (y1#1)
                                  (x2#1) -- (y2#1)
                                  (x3#1) -- (y3#1)
                                  (x4#1) -- (y4#1);

        \node at (h1#1) [above] {$h_{1}^{#1}$};
        \node at (h2#1) [below] {$h_{2}^{#1}$};
        \node at (h3#1) [below] {$h_{3}^{#1}$};
        \node at (h4#1) [below] {$h_{4}^{#1}$};
    }

    \drawTriangle{0}{1}
    \drawTriangle{32}{2}
  
    \node at (h21) [below] {$h_{2}^{1}$};
    \node at (h31) [left] {$h_{3}^{1}$};
    \node at (h22) [right] {$h_{2}^{2}$};
    \node at (h32) [below] {$h_{3}^{2}$};
    
    \draw [blue] (h11) -- (h12);
    \draw [blue] (h21) -- (h32);

    \drawDiamond{3}

    \node at (x13) [above] {$x_{1}^{3}$};
    \node at (y13) [above] {$y_{1}^{3}$};
    \node at (x23) [above] {$x_{2}^{3}$};
    \node at (y23) [below] {$y_{2}^{3}$};
    \node at (x33) [below] {$x_{3}^{3}$};
    \node at (y33) [below] {$y_{3}^{3}$};
    \node at (x43) [below] {$x_{4}^{3}$};
    \node at (y43) [above] {$y_{4}^{3}$};
    
    \draw [blue, bend right=10] (h31) to (h43);
    \draw [blue, bend left=10] (h22) to (h23);

    \foreach \i in {1, 2, 3}{
           \foreach \j in {1, 2, 3}{
              \solidnode{h\i\j};
              \solidnode{x\i\j};
              \solidnode{y\i\j};
            }
    }
              \solidnode{h43};
              \solidnode{x43};
              \solidnode{y43};

    \draw [red] ($(h23) + (8, -3)$) -- ($(h23) + (18, -3)$) node[right] {$1_{a}$}; 
    \draw [green] ($(h23) + (8, -6)$) -- ($(h23) + (18, -6)$) node[right] {$1_{b}$};
    \draw [blue] ($(h23) + (8, -9)$) -- ($(h23) + (18, -9)$) node[right] {$1_{c}$};
    \draw [brown] ($(h23) + (8, -12)$) -- ($(h23) + (18, -12)$) node[right] {$3_{a}$};
    \end{tikzpicture}
    \caption{A $(1, 1, 1, 3)$-packing edge-coloring of $G$ by substituting each vertex of $H$ with a triangle.}
    \label{M2}
\end{figure}

\noindent\textbf{Case 2: Some edges are replaced by strings of diamonds.}

Suppose $G$ is constructed by replacing some edges of $H$ with strings of diamonds in addition to substituting each vertex of $H$ with a triangle. For each cycle $C_{i}'$ in $\mathcal{F}'$, construct $C_{i}$ in $G$ as in Case 1. Color $C_{i}$ and $M$ as previously described in Case 1. For edges replaced by strings of diamonds, use the following method to color each string of diamonds.
\begin{enumerate}[leftmargin = 1.5cm,label = \textbf{Type~\arabic*}]
    \item \textbf{An edge in $M'$ of $H$ is replaced by a string of diamonds.} In this case, let the two sets of non-adjacent external edges of the string of diamonds be colored with $1_{a}$ and $1_{b}$, while all other edges in the string are colored with $1_{c}$.
    \item \textbf{An edge in the $2$-factor of $H$ is replaced by a string of diamonds.} Suppose the edge $y_{j}^{i} x_{j+1}^{i}$ in $\mathcal{F}$ is replaced by a string of diamonds. Let $c$ be its assigned color in Case 1.
    \begin{enumerate}[label = \textbf{Type~2.\arabic*}]
        \item If $c = 1_{a}$, color the two sets of non-adjacent external edges in the diamond string with $1_{b}$ and $1_{c}$, while assigning $1_{a}$ to all other edges.
        \item If $c = 1_{b}$, color the two sets of non-adjacent external edges in the diamond string with $1_{a}$ and $1_{c}$, while assigning $1_{b}$ to all other edges.
        \item If $c = 3_{a}$, color the two sets of non-adjacent external edges in the diamond string with $1_{c}$ and $1_{b}$ (or $1_{a}$), while assigning $1_{a}$ (or $1_{b}$) to all other edges in the diamond string, except for the edge at one endpoint which is assigned $3_{a}$.
    \end{enumerate}
\end{enumerate}

One can easily verify that this is indeed a $(1, 1, 1, 3)$-packing edge-coloring of $G$. (See \cref{M3} which depicts a $(1, 1, 1, 3)$-packing edge-coloring of $G$ obtained by replacing each vertex of $H$ with a triangle and some edges of $H$ with strings of diamonds.)
\end{proof}

\begin{figure}
    \centering
    \begin{tikzpicture}[line width=1pt, scale=0.19, every node/.style={font=\small}]

    \newcommand{\drawTriangle}[2]{
        \coordinate (h3#2) at (#1, 0);
        \coordinate (h2#2) at (#1 + 20, 0);
        \coordinate (h1#2) at ($(h3#2)!1!60:(h2#2)$);
        \coordinate (x1#2) at ($(h1#2)!0.25!(h3#2)$);
        \coordinate (y1#2) at ($(h1#2)!0.25!(h2#2)$);
        \coordinate (x2#2) at ($(h1#2)!0.75!(h2#2)$);
        \coordinate (y2#2) at ($(h2#2)!0.25!(h3#2)$);
        \coordinate (x3#2) at ($(h2#2)!0.75!(h3#2)$);
        \coordinate (y3#2) at ($(h1#2)!0.75!(h3#2)$);
        \coordinate (r2#2) at ($(y2#2)!0.25!(x3#2)$);
        \coordinate (r4#2) at ($(y2#2)!0.75!(x3#2)$);
        \coordinate (r1#2) at ($(r2#2)!0.7!-45:(r4#2)$);
        \coordinate (r3#2) at ($(r2#2)!0.7!45:(r4#2)$);
        \coordinate (w1#2) at ($(h2#2)!1!-60:(r1#2)$);
        \coordinate (w2#2) at ($(h2#2)!1!-60:(r2#2)$);
        \coordinate (w3#2) at ($(h2#2)!1!-60:(r3#2)$);
        \coordinate (w4#2) at ($(h2#2)!1!-60:(r4#2)$);
        \coordinate (g1#2) at ($(h3#2)!1!60:(r1#2)$);
        \coordinate (g2#2) at ($(h3#2)!1!60:(r2#2)$);
        \coordinate (g3#2) at ($(h3#2)!1!60:(r3#2)$);
        \coordinate (g4#2) at ($(h3#2)!1!60:(r4#2)$);
    \draw [brown] (w2#2) -- (x2#2);
        \draw [red] (x2#2) -- (h2#2)
                              (h3#2) -- (y3#2)
                              (x1#2) -- (h1#2)
                              (y2#2) -- (r2#2)
                              (r4#2) -- (x3#2)
                              (r1#2) -- (r3#2)
                              (g1#2) -- (g2#2)
                              (g3#2) -- (g4#2)
                              (y1#2) -- (w4#2)
                              (w1#2) -- (w3#2);
    \draw [green] (h1#2) -- (y1#2)
                              (h2#2) -- (y2#2)
                              (x3#2) -- (h3#2)
                              (r1#2) -- (r2#2)
                              (r3#2) -- (r4#2)
                              (g2#2) -- (x1#2)
                              (y3#2) -- (g4#2)
                              (g1#2) -- (g3#2)
                              (w1#2) -- (w2#2)
                              (w3#2) -- (w4#2);

        \draw [blue] (x1#2) -- (y1#2)
                                (x2#2) -- (y2#2)
                                (x3#2) -- (y3#2)
                                (r1#2) -- (r4#2)
                                (r2#2) -- (r3#2)
                                (w1#2) -- (w4#2)
                                (w2#2) -- (w3#2)
                                (g1#2) -- (g4#2)
                                (g2#2) -- (g3#2);

        \node at (h1#2) [above] {$h_{1}^{#2}$};
        \node at (x1#2) [left] {$x_{1}^{#2}$};
        \node at (y1#2) [right] {$y_{1}^{#2}$};
        \node at (x2#2) [right] {$x_{2}^{#2}$};
        \node at (y2#2) [below] {$y_{2}^{#2}$};
        \node at (x3#2) [below] {$x_{3}^{#2}$};
        \node at (y3#2) [left] {$y_{3}^{#2}$};
        \node at ($(r1#2)!1.3!(r3#2)$) {Type 2.1};
        \node [rotate = 60] at ($(g3#2)!1.3!(g1#2)$) {\small Type 2.2};
        \node [rotate = -60]at ($(w3#2)!1.3!(w1#2)$) {Type 2.3};
    }

    \newcommand{\drawDiamond}[1]{
        \coordinate (h1#1) at ($(h21)!0.5!(h32) + (0, -10)$);
        \coordinate (h3#1) at ($(h1#1) + (0,-20)$);
        \coordinate (h2#1) at ($(h1#1)!1!60:(h3#1)$);
        \coordinate (h4#1) at ($(h1#1)!1!-60:(h3#1)$);
        
        \coordinate (x1#1) at ($(h1#1)!0.25!(h4#1)$);
        \coordinate (y1#1) at ($(h1#1)!0.25!(h2#1)$);
        \coordinate (x2#1) at ($(h1#1)!0.75!(h2#1)$);
        \coordinate (y2#1) at ($(h2#1)!0.25!(h3#1)$);
        \coordinate (x3#1) at ($(h2#1)!0.75!(h3#1)$);
        \coordinate (y3#1) at ($(h3#1)!0.25!(h4#1)$);
        \coordinate (x4#1) at ($(h3#1)!0.75!(h4#1)$);
        \coordinate (y4#1) at ($(h1#1)!0.75!(h4#1)$);

        \draw [blue] (h1#1) -- (h3#1);
        \draw (h1#1) -- (h2#1) -- (h3#1) -- (h4#1) -- cycle;
        \draw [red]   (x1#1) -- (h1#1)
                            (y1#1) -- (x2#1)
                            (h2#1) -- (y2#1)
                            (x3#1) -- (h3#1)
                            (y3#1) -- (x4#1)
                            (h4#1) -- (y4#1);
                            
      \draw [green] (h1#1) -- (y1#1)
                            (x2#1) -- (h2#1)
                            (y2#1) -- (x3#1)
                            (h3#1) -- (y3#1)
                            (x4#1) -- (h4#1)
                            (y4#1) -- (x1#1);
                            
        \draw [blue] (x1#1) -- (y1#1)
                                  (x2#1) -- (y2#1)
                                  (x3#1) -- (y3#1)
                                  (x4#1) -- (y4#1);

        \node at (h1#1) [above] {$h_{1}^{#1}$};
        \node at (h2#1) [below] {$h_{2}^{#1}$};
        \node at (h3#1) [below] {$h_{3}^{#1}$};
        \node at (h4#1) [below] {$h_{4}^{#1}$};
    }

    \drawTriangle{0}{1}
    \drawTriangle{32}{2}
  
    \node at (h21) [below] {$h_{2}^{1}$};
    \node at (h31) [left] {$h_{3}^{1}$};
    \node at (h22) [right] {$h_{2}^{2}$};
    \node at (h32) [below] {$h_{3}^{2}$};
    
    \coordinate (b41) at ($(h11)!2/9!(h12)$);
    \coordinate (b21) at ($(h11)!4/9!(h12)$);
    \coordinate (b11) at ($(b41)!0.7!45:(b21)$);
    \coordinate (b31) at ($(b41)!0.7!-45:(b21)$);
    \coordinate (b42) at ($(h11)!5/9!(h12)$);
    \coordinate (b22) at ($(h11)!7/9!(h12)$);
    \coordinate (b12) at ($(b42)!0.7!45:(b22)$);
    \coordinate (b32) at ($(b42)!0.7!-45:(b22)$);
    \draw [red]   (b11) -- (b21)
                            (b31) -- (b41)
                            (b12) -- (b22)
                            (b32) -- (b42);
    \draw [green] (b11) -- (b41)
                              (b21) -- (b31)
                              (b12) -- (b42)
                              (b22) -- (b32);
    \draw [blue] (h21) -- (h32)
                            (h11) -- (b41)
                            (b11) -- (b31)
                            (b21) -- (b42)
                            (b12) -- (b32)
                            (b22) -- (h12);

    \drawDiamond{3}

    \node at (x13) [above] {$x_{1}^{3}$};
    \node at (y13) [above] {$y_{1}^{3}$};
    \node at (x23) [above] {$x_{2}^{3}$};
    \node at (y23) [below] {$y_{2}^{3}$};
    \node at (x33) [below] {$x_{3}^{3}$};
    \node at (y33) [below] {$y_{3}^{3}$};
    \node at (x43) [below] {$x_{4}^{3}$};
    \node at (y43) [above] {$y_{4}^{3}$};
    
    \draw [blue, bend right=10] (h31) to (h43);
    \draw [blue, bend left=10] (h22) to (h23);

    \foreach \i in {1, 2, 3}{
           \foreach \j in {1, 2, 3}{
              \solidnode{h\i\j};
              \solidnode{x\i\j};
              \solidnode{y\i\j};
            }
    }
    \foreach \i in {1, 2, 3, 4}{
           \foreach \j in {1, 2}{
              \solidnode{r\i\j};
              \solidnode{g\i\j};
              \solidnode{b\i\j};
              \solidnode{w\i\j};
            }
    }
              \solidnode{h43};
              \solidnode{x43};
              \solidnode{y43};

    \node at ($(b11)!0.5!(b12)$) {Type 1};
    \draw [red] ($(h23) + (8, -3)$) -- ($(h23) + (18, -3)$) node[right] {$1_{a}$}; 
    \draw [green] ($(h23) + (8, -6)$) -- ($(h23) + (18, -6)$) node[right] {$1_{b}$};
    \draw [blue] ($(h23) + (8, -9)$) -- ($(h23) + (18, -9)$) node[right] {$1_{c}$};
    \draw [brown] ($(h23) + (8, -12)$) -- ($(h23) + (18, -12)$) node[right] {$3_{a}$};
    \end{tikzpicture}
    \caption{A $(1, 1, 1, 3)$-packing edge-coloring of $G$ by replacing each vertex of $H$ with a triangle and some edges of $H$ with strings of diamonds.}
    \label{M3}
\end{figure}

To extend the $(1, 1, 1, 3)$-packing edge-coloring to claw-free cubic graphs containing bridges, we proceed as follows: Let $G$ be a connected claw-free cubic graph with bridge set $B(G) = \{ e_{1}, \dots, e_{b} \}$. Removing $B(G)$ decomposes $G$ into components $\{G_{i}\}_{0 \leq i \leq b}$, where each $G_{i}$ is a connected induced subgraph of $G$. Define a new graph $T_{G}$, referred to as the bridge tree of $G$, as follows: The vertex set of $T_{G}$ is $V(T_{G}) = \{ g_{0}, g_{1}, \dots, g_{b} \}$, where $g_{i}$ corresponds to component $G_{i}$ of $G - B(G)$. The edge set of $T_{G}$ is 
\begin{equation*}
E(T_{G}) = \{ g_{i}g_{j} \mid \text{there exists a bridge $e \in B(G)$ that connects $G_{i}$ and $G_{j}$}\}.
\end{equation*}
Since each edge in $B(G)$ is a bridge, $T_{G}$ is a tree. 

The following lemma is essential for the subsequent proof process.

\begin{lemma}[Bre\v{s}ar, Kuenzel and Rall \cite{MR4878746}]\label{b}
Let $G$ be a connected claw-free cubic graph with $|B(G)| \geq 1$, and let $G_{i}$ be the components of $G - B(G)$. Then $G_{i}$ is $2$-edge-connected and satisfies one of the following.
\begin{itemize}
    \item[(1)] $G_{i}$ is isomorphic to $K_{3}$ (\ie $\Delta(G_{i}) = 2$);
    \item[(2)] $G_{i}$ is isomorphic to a diamond $D$ (\ie $\Delta(G_{i}) = 3$ and $|V(G_{i})| = 4$);
    \item[(3)] $G_{i}$ is isomorphic to a $2$-edge-connected graph with $\Delta(G_{i}) = 3$ and $|V(G_{i})| \geq 5$. Moreover, $G_{i}$ is a leaf of $T_{G}$ if and only if it contains exactly one degree-2 vertex.
\end{itemize}
\end{lemma}

Fix a longest path $P$ in $T_{G}$ of length $d = \operatorname{diam}(T_{G})$, and root $T_{G}$ at a leaf $g_{0}$ of $P$. Partition $V(T_{G})$ into levels $A_{i} = \{ g_{j} \in V(T_{G}) \mid d_{T_{G}} (g_{0}, g_{j}) = i \}$. By construction:
\begin{itemize}[leftmargin=*]
    \item $A_{0} = \{ g_{0} \}$;
    \item Each $g_{j} \in A_{i}$ ($i > 0$) has a unique parent $g_{k} \in A_{i-1}$;
    \item Every vertex of $A_{d}$ is a leaf in $T_{G}$.
\end{itemize}

Consequently, if $g_{j} \in A_{i}$ ($i > 0$), there exists a unique degree-2 vertex $p \in V(G_{j})$ with exactly one neighbor $q \in V(G_{k})$, where $g_{k}$ is the parent of $g_{j}$ in $T_{G}$. In this context, $q$ is referred to as the \emph{up-neighbor} of $p$.

When $G_{i} \cong K_{3}$ or $G_{i} \cong D$, a desired coloring is straightforward. We therefore focus on constructing a $2$-edge-connected claw-free cubic graph $\widetilde{G_{i}}$ when $G_{i}$ satisfies conclusion (3) of \cref{b}. Assume $G_{i}$ is a $2$-edge-connected graph with maximum degree $3$, where $|V(G_{i})| \geq 5$. Let $V_{i} = \{ v_{1}^{i}, v_{2}^{i}, \dots, v_{r_{i}}^{i} \}$ be the set of all degree-2 vertices in $G_{i}$, where $v_{1}^{i}$ is uniquely required to have an up-neighbor. For each $v_{j}^{i} \in V_{i}$, where $1 \leq j \leq r_{i}$, let $u_{j}^{i}$ and $w_{j}^{i}$ denote the two neighbors of $v_{j}^{i}$ in $G_{i}$.

\begin{claim}\label{bb}
Every vertex in $V_{i}$ lies on a triangle of $G_{i}$, \ie $u_{j}^{i} w_{j}^{i} \in E(G_{i})$. Consequently, $V_{i}$ is an independent set.
\end{claim}

\begin{proof}
Since $G$ is claw-free and $v_{j}^{i}$ is incident with a bridge, it follows that $v_{j}^{i}$ must lie on a triangle of $G_{i}$.

Assume without loss of generality that $v_{1}^{i}$ and $v_{2}^{i}$ are adjacent. As every vertex in $V_{i}$ lies on a triangle of $G_{i}$, we have that $v_{1}^{i}$ and $v_{2}^{i}$ share a common neighbor $z_{1}^{i}$. If $d_{G_{i}}(z_{1}^{i}) = 2$, this contradicts the condition $|V(G_{i})| \geq 5$. If $d_{G_{i}}(z_{1}^{i}) = 3$, $G_{i}$ contains a bridge, which is also a contradiction.
\end{proof}

Let $s_{j}^{i}$ be the neighbor of $u_{j}^{i}$ distinct from $v_{j}^{i}$ and $w_{j}^{i}$, and let $b_{j}^{i}$ be the neighbor of $w_{j}^{i}$ distinct from $u_{j}^{i}$ and $v_{j}^{i}$.

\begin{claim}\label{tt}
$s_{j}^{i} \neq b_{j}^{i}$, and $d_{G_{i}}(s_{j}^{i}) = d_{G_{i}}(b_{j}^{i}) = 3$.
\end{claim}

\begin{proof}
Assume $s_{j}^{i} = b_{j}^{i}$. In this case, $G_{i}$ must either contain a bridge or be a diamond, leading to a contradiction. Since $d_{G_{i}}(u_{j}^{i}) = 3$ and $u_{j}^{i}b_{j}^{i} \notin E(G)$, if $d_{G_{i}}(b_{j}^{i}) = 2$, \cref{bb} implies that $b_{j}^{i}$ must lie on a triangle of $G_{i}$, which is a contradiction. By symmetry, $s_{j}^{i}$ must also have degree $3$ in $G_{i}$.
\end{proof}

Now, we construct a $2$-edge-connected claw-free cubic graph $\widetilde{G_{i}}$ from $G_{i}$. Consider two cases based on the parity of $r_{i}$.

\begin{case}[\textbf{\bm{$r_{i}$} even}]\label{c}
Construct $\widetilde{G_{i}}$ from $G_{i}$ by adding the edges $v_{j}^{i}v_{j+1}^{i}$ for all odd $j$ with $1 \leq j \leq r_{i} - 1$.
\end{case}

\begin{case}[\textbf{\bm{$r_{i}$} odd}]\label{cc}
By \cref{bb,tt}, we have $u_{1}^{i}w_{1}^{i} \in E(G_{i})$, $s_{1}^{i} \neq b_{1}^{i}$, and $d_{G_{i}}(s_{1}^{i}) = d_{G_{i}}(b_{1}^{i}) = 3$.
Suppose $s_{1}^{i}b_{1}^{i} \in E(G_{i})$. Since $G$ is claw-free, it follows that $s_{1}^{i}$ and $b_{1}^{i}$ share a common neighbor $z_{1}^{i}$. If $d_{G_{i}}(z_{1}^{i}) = 2$, this contradicts $r_{i}$ being odd (as it would cause $r_{i} = 2$). If $d_{G_{i}}(z_{1}^{i}) = 3$, then $G_{i}$ would contain a bridge, which is also a contradiction. Thus, $s_{1}^{i}b_{1}^{i} \notin E(G_{i})$. Let $\widetilde{G_{i}}$ be the graph obtained from $G_{i} - \{v_{1}^{i}, u_{1}^{i}, w_{1}^{i} \}$ by adding edge $s_{1}^{i}b_{1}^{i}$ and edges $v_{j}^{i}v_{j+1}^{i}$ for all even $j$ with $2 \leq j \leq r_{i} - 1$.
\end{case}

In \cite{MR4878746}, it was proved that $\widetilde{G_{i}}$ is a $2$-edge-connected claw-free cubic graph, and if $r_{i}$ is odd, then $s_{1}^{i}b_{1}^{i}$ does not lie on a triangle of $\widetilde{G_{i}}$ unless $\widetilde{G_{i}} \cong K_{4}$.

\begin{theorem}
Every connected claw-free cubic graph is $(1, 1, 1, 3)$-packing edge-colorable.
\end{theorem}

\begin{proof}
Let $G$ be a connected claw-free cubic graph with $|B(G)| = b \geq 0$, where $B(G)$ denotes its bridge set. Let $G_{i}$ denote the components of $G - B(G)$.
When $b = 0$, \cref{m} provides a $(1, 1, 1, 3)$-packing edge-coloring of $G$. For $b > 0$, consider the tree $T_{G}$ defined by $V(T_{G}) = \{g_{0}, g_{1}, \dots, g_{b}\}$ and $E(T_{G}) = \{g_{i}g_{j} \mid \text{there exists an edge $e \in B(G)$ that connects $G_{i}$ and $G_{j}$}\}$. Choose a path $P$ in $T_{G}$ of length $d = \operatorname{diam}(T_{G})$ and let $g_{0}$ be a leaf of $P$.

First, we consider the component $G_{0}$. Clearly, $g_{0}$ is also a leaf of $T_{G}$. By \cref{b}, $G_{0}$ contains exactly one degree-2 vertex $v_{1}^{0}$, and in this case, $r_{0} = 1$. By \cref{bb,tt}, $u_{1}^{0}w_{1}^{0} \in E(G_{0})$, $s_{1}^{0} \neq b_{1}^{0}$ and $s_{1}^{0}b_{1}^{0} \notin E(G_{0})$.
We construct $\widetilde{G_{0}}$ as follows: apply the method as Case~\ref{cc} to obtain the graph $\widetilde{G_{0}}$ from $G_{0} - \{v_{1}^{0}, u_{1}^{0}, w_{1}^{0}\}$ by adding the edge $s_{1}^{0}b_{1}^{0}$. The resulting graph $\widetilde{G_{0}}$ is a $2$-edge-connected claw-free cubic graph.

\begin{claim}\label{mm}
$\widetilde{G_{0}}$ admits a $(1, 1, 1, 3)$-packing edge-coloring $\varphi'$ such that $\varphi'(s_{1}^{0}b_{1}^{0}) = 1_{a}$, and no edge incident with $s_{1}^{0}$ or $b_{1}^{0}$ is colored with $3_{a}$.
\end{claim}

\begin{proof}
Note that the edge $s_{1}^{0}b_{1}^{0}$ is not on a triangle of $\widetilde{G_{0}}$ unless $\widetilde{G_{0}} \cong K_{4}$. If $\widetilde{G_{0}} \cong K_{4}$ or $\widetilde{G_{0}}$ is a ring of diamonds, it is straightforward to verify that there exists a proper edge-coloring using colors $1_{a}, 1_{b}, 1_{c}$ such that $s_{1}^{0}b_{1}^{0}$ is colored with $1_{a}$. So we may assume that $s_{1}^{0}b_{1}^{0}$ is not on a triangle of $\widetilde{G_{0}}$, and $\widetilde{G_{0}}$ is not a ring of diamonds. By \cref{thm:oum}, let $\widetilde{G_{0}}$ be constructed from a $2$-edge-connected cubic multigraph $H$ by substituting each vertex of $H$ with a triangle, and possibly replacing some edges of $H$ with strings of diamonds. Let $s_{1}^{0}b_{1}^{0}$ correspond to an edge $e$ in $H$. By \cref{thm:2factor}, $H$ has a $2$-factor $\mathcal{F}'$ containing $e$, where $e \in C_{i}'$. Then $s_{1}^{0}b_{1}^{0}$ corresponds to $y_{j}^{i}x_{j+1}^{i}$, where $y_{j}^{i}x_{j+1}^{i}$ may or may not be replaced by a string of diamonds, and $1 \leq j \leq m_{i}$.

We consider two cases:
\begin{itemize}
    \item[(1)] \textbf{$m_{i}$ is even.} By the coloring process described in the proof of \cref{m}, $\widetilde{G_{0}}$ admits a $(1, 1, 1, 3)$-packing edge-coloring $\varphi'$ with $\varphi'(s_{1}^{0}b_{1}^{0}) = 1_{a}$.
    \item[(2)] \textbf{$m_{i}$ is odd.} If necessary, we can relabel the vertices on the cycle $x_{1}^{i}h_{1}^{i}y_{1}^{i}x_{2}^{i}h_{2}^{i}y_{2}^{i}\dots x_{m_{i}}^{i}h_{m_{i}}^{i}y_{m_{i}}^{i}x_{1}^{i}$ such that, from the coloring process described in the proof of \cref{m}, $e$ is colored with $1_{a}$, and correspondingly $s_{1}^{0}b_{1}^{0}$ is also colored with $1_{a}$.
\end{itemize}

This completes the proof of \cref{mm}.
\end{proof}

Let $\varphi'$ be an edge-coloring from \cref{mm}. By symmetry, let $b_{1}^{0}w_{1}^{0}$ be the edge in $\{b_{1}^{0}w_{1}^{0}, s_{1}^{0}u_{1}^{0}\}$ farthest from the edges colored with $3_{a}$ in $\widetilde{G_{0}}$. Now, define a coloring $\varphi$ of $G_{0}$ as follows: let $\varphi(e) = \varphi'(e)$ for all $e \in E(G_{0} - \{v_{1}^{0}, u_{1}^{0}, w_{1}^{0}\})$, color the edges of the path $s_{1}^{0}u_{1}^{0}v_{1}^{0}w_{1}^{0}$ alternately with $1_{a}$ and $1_{b}$ in sequence, and assign $\varphi(b_{1}^{0}w_{1}^{0}) = 3_{a}$ and $\varphi(u_{1}^{0}w_{1}^{0}) = 1_{c}$.

Next, we consider each $G_{i}$, where $i \geq 1$. If $G_{i} \cong K_{3}$, then the edges of $G_{i}$ are assigned the colors $1_{a}, 1_{b}$ and $1_{c}$, respectively. If $G_{i} \cong D$, then the two sets of non-adjacent external edges and one internal edge are assigned the colors $1_{a}, 1_{b}$, and $1_{c}$, respectively.

\begin{claim}
Every $2$-edge-connected $G_{i}$ with $\Delta(G_{i}) = 3$ and $|V(G_{i})| \geq 5$ admits a $(1, 1, 1, 3)$-packing edge-coloring.
\end{claim}

\begin{proof}
We consider the cases based on the parity of $r_{i}$.
\begin{itemize}
    \item[(1)] \textbf{$r_{i}$ is even.} We construct $\widetilde{G_{i}}$ via Case~\ref{c}. By \cref{m}, $\widetilde{G_{i}}$ has a $(1, 1, 1, 3)$-packing edge-coloring $\varphi'$. Let $\varphi$ be a coloring of $G_{i}$, and $\varphi(e) = \varphi'(e)$ for all $e\in E(G_{i})$. Clearly, $\varphi$ is a $(1, 1, 1, 3)$-packing edge-coloring.
    \item[(2)] \textbf{$r_{i}$ is odd.} We construct $\widetilde{G_{i}}$ via Case~\ref{cc}. By \cref{m}, $\widetilde{G_{i}}$ has a $(1, 1, 1, 3)$-packing edge-coloring $\varphi'$. By symmetry, let $b_{1}^{i}w_{1}^{i}$ be the edge in $\{b_{1}^{i}w_{1}^{i}, s_{1}^{i}u_{1}^{i}\}$ farthest from the edges colored with $3_{a}$ in $\widetilde{G_{i}}$. Let $\varphi$ be a coloring of $G_{i}$, and $\varphi(e) = \varphi'(e)$ for all $e \in E(G_{i})$.
	Color the edges of the path $s_{1}^{i}u_{1}^{i}v_{1}^{i}w_{1}^{i}$ alternately with colors $1_{a}$ and $1_{b}$ in sequence, and assign $\varphi(b_{1}^{i}w_{1}^{i}) = 3_{a}$ and $\varphi(u_{1}^{i}w_{1}^{i}) = 1_{c}$. From the proof of \cref{mm}, this yields a valid $(1, 1, 1, 3)$-packing edge-coloring. \qedhere
\end{itemize}
\end{proof}

\begin{claim}
For the coloring described above on $G_{i}$, all edges incident with degree-2 vertices in $G_{i}$ use colors from $\{1_{a}, 1_{b}, 1_{c}\}$.
\end{claim}

\begin{proof}
For any degree-2 vertex $v_{j}^{i}$ in $G_{i}$, let $u_{j}^{i}$ and $w_{j}^{i}$ denote the two neighbors of $v_{j}^{i}$ in $G_{i}$. Since $G$ is claw-free, we have $u_{j}^{i}w_{j}^{i} \in E(G_{i})$. According to the coloring process for $G_{i}$, edges colored with $3_{a}$ cannot lie on a triangle. The coloring process explicitly assigns only $\{1_{a}, 1_{b}, 1_{c}\}$ to edges adjacent to $v_{j}^{i}$, satisfying the claim.
\end{proof}

The colors in $\{1_{a}, 1_{b}, 1_{c}\}$ can be freely permuted without affecting the validity of the packing condition on $G_{i}$. Specifically, if we perform a permutation on the edges that have already been assigned colors, for example, by changing all edges assigned $1_{a}$ to $1_{b}$ and all edges assigned $1_{b}$ to $1_{a}$, this operation does not cause any contradictions or conflicts. This is because the permutation only reassigns the colors of the edges without altering the structure of the edges or the coloring requirements of the graph.

We perform a coloring of $T_{G}$ using Breadth-First Search (BFS). First, assign $\{ 1_{a}, 1_{b}, 1_{c}, 3_{a} \}$ to the edges of $G_{0}$ (root component). For each edge $pq$ connecting $A_{0}$ and $A_{1}$, where $q$ is the up-neighbor of $p$, the edge $pq$'s color is determined by $G_{0}$'s coloring. Note that each component $G_{i} \neq G_{0}$ contains exactly one vertex having an up-neighbor. If necessary, we can permute the colors $1_{a}, 1_{b}, 1_{c}$ in $G_{i}$ to allow $pq$ to use the color missing at $q$. 
Next, by appropriately permuting $\{1_{a}, 1_{b}, 1_{c}\}$, assign colors to the edges of $G_{i}$ in $A_{1}$. The color of the edges between $A_{1}$ and $A_{2}$ is then determined by the coloring of $G_{i}$ corresponding to a vertex $g_{i} \in A_{1}$. Continuing this process iteratively, we can obtain a $(1, 1, 1, 3)$-packing edge-coloring of $G$.
\end{proof}


\begin{thebibliography}{10}

\bibitem{MR4145729}
B.~Bre\v{s}ar, N.~Gastineau and O.~Togni, Packing colorings of subcubic
  outerplanar graphs, Aequationes Math. 94~(5) (2020) 945--967.

\bibitem{MR3600792}
B.~Bre\v{s}ar, S.~Klav\v{z}ar, D.~F. Rall and K.~Wash, Packing chromatic
  number, {$(1,1,2,2)$}-colorings, and characterizing the {P}etersen graph,
  Aequationes Math. 91~(1) (2017) 169--184.

\bibitem{MR4878746}
B.~Bre\v{s}ar, K.~Kuenzel and D.~F. Rall, Claw-free cubic graphs are $(1, 1, 2,
  2)$-colorable, Discrete Math. 348~(8) (2025) 114477.

\bibitem{MR3053595}
J.-L. Fouquet and J.-M. Vanherpe, On parsimonious edge-colouring of graphs with
  maximum degree three, Graphs Combin. 29~(3) (2013) 475--487.

\bibitem{MR3508758}
N.~Gastineau and O.~Togni, {$S$}-packing colorings of cubic graphs, Discrete
  Math. 339~(10) (2016) 2461--2470.

\bibitem{MR3944589}
N.~Gastineau and O.~Togni, On {$S$}-packing edge-colorings of cubic graphs,
  Discrete Appl. Math. 259 (2019) 63--75.

\bibitem{MR2993523}
W.~Goddard and H.~Xu, The {$S$}-packing chromatic number of a graph, Discuss.
  Math. Graph Theory 32~(4) (2012) 795--806.

\bibitem{MR3163180}
W.~Goddard and H.~Xu, A note on {$S$}-packing colorings of lattices, Discrete
  Appl. Math. 166 (2014) 255--262.

\bibitem{MR4520054}
H.~Hocquard, D.~Lajou and B.~Lu\v{z}ar, Between proper and strong
  edge-colorings of subcubic graphs, J. Graph Theory 101~(4) (2022) 686--716.

\bibitem{MR4273008}
A.~Kostochka and X.~Liu, Packing {$(1,1,2,4)$}-coloring of subcubic outerplanar
  graphs, Discrete Appl. Math. 302 (2021) 8--15.

\bibitem{Li2024}
S.~Li, Y.~Li and X.~Liu, Packing edge-colorings of subcubic outerplanar graphs,
  arXiv: 2411.05720  (2024) \url{https://doi.org/10.48550/arXiv.2411.05720}.

\bibitem{MR4660817}
X.~Liu, M.~Santana and T.~Short, Every subcubic multigraph is $(1,
  2^7)$-packing edge-colorable, J. Graph Theory 104~(4) (2023) 851--885.

\bibitem{Liu2024}
X.~Liu and G.~Yu, On the $(1^2,2^4)$-packing edge-coloring of subcubic graphs,
  arXiv: 2402.18353  (2024) \url{https://doi.org/10.48550/arXiv.2402.18353}.

\bibitem{MR2788679}
S.-i. Oum, Perfect matchings in claw-free cubic graphs, Electron. J. Combin.
  18~(1) (2011) Paper 62.

\bibitem{Payan1977}
C.~Payan, Sur quelques probl{\`e}mes de couverture et de couplage en
  combinatoire, Ph.D. thesis, {Institut National Polytechnique de Grenoble -
  INPG ; Universit{\'e} Joseph-Fourier - Grenoble I} (1977).

\bibitem{MR317999}
J.~Plesn\'{\i}k, Connectivity of regular graphs and the existence of
  {$1$}-factors, Mat. \v{C}asopis Sloven. Akad. Vied 22 (1972) 310--318.

\end{thebibliography}
\end{document}